\documentclass[oneside,english]{amsart}
\usepackage{lmodern}
\usepackage[T1]{fontenc}
\usepackage[latin9]{inputenc}
\usepackage{geometry}
\usepackage{color}
\usepackage{units}
\usepackage{amsbsy}
\usepackage{amstext}
\usepackage{amsthm}
\usepackage{amssymb}

\makeatletter
\numberwithin{equation}{section}
\numberwithin{figure}{section}
\theoremstyle{plain}
\newtheorem{thm}{\protect\theoremname}
\theoremstyle{plain}
\newtheorem{cor}[thm]{\protect\corollaryname}
\theoremstyle{plain}
\newtheorem{lem}[thm]{\protect\lemmaname}
\theoremstyle{remark}

\usepackage{etoolbox}
\patchcmd{\thmhead}{(#3)}{#3}{}{}

\makeatother

\usepackage{babel}
\providecommand{\corollaryname}{Corollary}
\providecommand{\lemmaname}{Lemma}
\providecommand{\remarkname}{Remark}
\providecommand{\theoremname}{Theorem}

\begin{document}
\global\long\def\e{e}%
\global\long\def\V{{\rm Vol}}%
\global\long\def\bs{\boldsymbol{\sigma}}%
\global\long\def\bx{\mathbf{x}}%
\global\long\def\by{\mathbf{y}}%
\global\long\def\bv{\mathbf{v}}%
\global\long\def\bu{\mathbf{u}}%
\global\long\def\bn{\mathbf{n}}%
\global\long\def\bY{\mathbf{Y}}%
\global\long\def\grad{\nabla_{sp}}%
\global\long\def\Hess{\nabla_{sp}^{2}}%
\global\long\def\lp{\Delta_{sp}}%
\global\long\def\gradE{\nabla_{\text{Euc}}}%
\global\long\def\HessE{\nabla_{\text{Euc}}^{2}}%
\global\long\def\HessEN{\hat{\nabla}_{\text{Euc}}^{2}}%
\global\long\def\ddq{\frac{d}{dR}}%
\global\long\def\qs{q_{\star}}%
\global\long\def\qss{q_{\star\star}}%
\global\long\def\lm{\lambda_{min}}%
\global\long\def\Es{E_{\star}}%
\global\long\def\EH{E_{\Hess}}%
\global\long\def\Esh{\hat{E}_{\star}}%
\global\long\def\ds{d_{\star}}%
\global\long\def\Cs{\mathscr{C}_{\star}}%
\global\long\def\nh{\boldsymbol{\hat{\mathbf{n}}}}%
\global\long\def\BN{\mathbb{B}^{N}}%
\global\long\def\ii{\mathbf{i}}%
\global\long\def\SN{\mathbb{S}^{N-1}}%
\global\long\def\SNq{\mathbb{S}^{N-1}(q)}%
\global\long\def\SNqd{\mathbb{S}^{N-1}(q_{d})}%
\global\long\def\SNqp{\mathbb{S}^{N-1}(q_{P})}%
\global\long\def\nd{\nu^{(\delta)}}%
\global\long\def\nz{\nu^{(0)}}%
\global\long\def\cls{c_{LS}}%
\global\long\def\qls{q_{LS}}%
\global\long\def\dls{\delta_{LS}}%
\global\long\def\E{\mathbb{E}}%
\global\long\def\P{\mathbb{P}}%
\global\long\def\R{\mathbb{R}}%
\global\long\def\spp{{\rm Supp}(\mu_{P})}%
\global\long\def\indic{\mathbf{1}}%
\global\long\def\lsc{\mu_{{\rm sc}}}%
\newcommand{\SNarg}[1]{\mathbb S^{N-1}(#1)}
\global\long\def\se{s(E)}%
\global\long\def\ses{s(\Es)}%
\global\long\def\so{s(0)}%
\global\long\def\sef{s(E_{f})}%
\global\long\def\seinf{s(E_{\infty})}%
\global\long\def\L{\mathcal{L}}%
\global\long\def\gflow#1#2{\varphi_{#2}(#1)}%
\newcommand{\corO}{\textcolor{red}}
\newcommand{\corE}{\textcolor{blue}}

\title{Concentration of the complexity of spherical pure $p$-spin models at
arbitrary energies}
\author{Eliran Subag}
\address{\tiny{Eliran Subag, Department of Mathematics, Weizmann Institute of Science, Rehovot 76100, Israel.}}
\email{eliran.subag@weizmann.ac.il}

\author{Ofer Zeitouni}
\address{\tiny{Ofer Zeitouni, Department of Mathematics, Weizmann Institute of Science, Rehovot 76100, Israel.}}
\email{ofer.zeitouni@weizmann.ac.il}

\thanks{This project has received funding from the European Research Council (ERC) under the European Union's Horizon 2020 research and innovation programme (grant agreement No. 692452) and from the Israel Science Foundation
	(grant agreement No. 2055/21).}

\begin{abstract}
  We consider critical points of the
  spherical \emph{pure }$p$-spin spin glass
model with Hamiltonian
$H_{N}\left(\boldsymbol{\sigma}\right)=\frac{1}{N^{\left(p-1\right)/2}}\sum_{i_{1},...,i_{p}=1}^{N}J_{i_{1},...,i_{p}}\sigma_{i_{1}}\cdots\sigma_{i_{p}}$,
where $\boldsymbol{\sigma}=\left(\sigma_{1},...,\sigma_{N}\right)\in
\SN:=\left\{ \boldsymbol{\sigma}\in\mathbb{R}^{N}:\,\left\Vert \boldsymbol{\sigma}\right\Vert _{2}=\sqrt{N}\right\} $
and $J_{i_{1},...,i_{p}}$ are i.i.d standard normal variables.
Using a second moment analysis, we prove that for $p\geq 32$ and
any $E>-\Es$, where
$\Es$ is the (normalized) ground state,
the ratio of the number
of critical points $\boldsymbol{\sigma}$ with
$H_N(\boldsymbol{\sigma})\leq NE$ and its expectation
asymptotically concentrates
at $1$.
This extends to arbitrary $E$ a similar conclusion of \cite{2nd}.
\end{abstract}
\maketitle
\section{Introduction}

The Hamiltonian of the spherical \emph{pure }$p$-spin spin glass
model is given by
\begin{equation}
H_{N}\left(\boldsymbol{\sigma}\right):=\frac{1}{N^{\left(p-1\right)/2}}\sum_{i_{1},...,i_{p}=1}^{N}J_{i_{1},...,i_{p}}\sigma_{i_{1}}\cdots\sigma_{i_{p}},\label{eq:Hamiltonian}
\end{equation}
where $\boldsymbol{\sigma}=\left(\sigma_{1},...,\sigma_{N}\right)$
belongs to the sphere  $\SN:=\left\{ \boldsymbol{\sigma}\in\mathbb{R}^{N}:\,\left\Vert \boldsymbol{\sigma}\right\Vert _{2}=\sqrt{N}\right\} $,
and $J_{i_{1},...,i_{p}}$ are i.i.d standard normal variables. The
model was introduced by Crisanti and Sommers \cite{Crisanti1992}
for general $p\geq2$ as a continuous variant of the same models with
Ising spins, while for $p=2$ it was already considered by Kosterlitz,
Thouless and Jones \cite{Kosterlitz} in analogy to the Sherrington-Kirkpatrick
model \cite{SK75}.

For any subset $B\subset\mathbb{R}$, let $\mbox{Crt}_{N}\left(B\right)$
be the number of critical points of $H_{N}$, at which it attains
a value in $NB=\left\{ Nx:\,x\in B\right\} $,
\[
\mbox{Crt}_{N}\left(B\right):=\#\left\{ \boldsymbol{\sigma}\in\SN:\,\nabla H_{N}\left(\boldsymbol{\sigma}\right)=0,\,H_{N}\left(\boldsymbol{\sigma}\right)\in NB\right\} ,
\]
where $\nabla H_{N}\left(\boldsymbol{\sigma}\right)$ denotes the
spherical gradient of $H_{N}\left(\boldsymbol{\sigma}\right)$. We refer
to $\mbox{Crt}_{N}\left( (-\infty,u)\right)$ as the \textit{complexity} of
$H_N$ at energies below $Nu$. In
their seminal work \cite{A-BA-C}, Auffinger, Ben Arous, and {\v{C}}ern{\'y}
proved that
\begin{equation}
\lim_{N\to\infty}\frac{1}{N}\log\E{\rm Crt}_{N}\left(\left(-\infty,u\right)\right)=\Theta_{p}(u),\label{eq:complexity}
\end{equation}
for
\[
\Theta_{p}(u)=\begin{cases}
\frac{1}{2}\log(p-1)-\frac{p-2}{4(p-1)}u^{2}-J(u) & \text{if }u\leq-E_{\infty},\\
\frac{1}{2}\log(p-1)-\frac{p-2}{4(p-1)}u^{2} & \text{if }-E_{\infty}\leq u\leq0,\\
\frac{1}{2}\log(p-1) & \text{if }u\geq0,
\end{cases}
\]
where $E_{\infty}=E_{\infty}(p)=2\sqrt{\frac{p-1}{p}}$ and for $u\leq-E_{\infty}$,
\[
J(u)=-\frac{u}{E_{\infty}^{2}}\sqrt{u^{2}-E_{\infty}^{2}}-\log\big(-u+\sqrt{u^{2}-E_{\infty}^{2}}\big)+\log E_{\infty}.
\]
A similar result was also proved in \cite{A-BA-C} for the expectation
of the number of critical points of any fixed index $k\geq0$. In
particular, it was shown that for any $k\geq0$ and $\epsilon>0$,
with probability going to $1$ as $N\to\infty$, there are no critical
points $\bs$ of index $k$ with $H_{N}(\bs)\in N(-E_{\infty}+\epsilon,\infty)$.

Define the ground-state energy
\[
-\Es:=\lim_{N\to\infty}\frac{1}{N}\E\min_{\bs\in\SN}H_{N}(\bs).
\]
The value of $\Es$ can be computed from the Parisi formula \cite{Talag,Chen}
for the free energy by letting the temperature go to zero or directly
from its zero temperature analogue \cite{ChenSen,JagannathTobascoLowTemp};
from the TAP representation \cite{FElandscape,SubagPspinTAP} for the free energy;
or from the analysis of critical points, as we will now explain. Define $E_{0}=E_{0}(p)>0$
as the unique value such that $\Theta_{p}(-E_{0})=0$. Then, by Markov's
inequality, $\Es\leq E_{0}$. In fact, it was proved in \cite{A-BA-C}
that $\Es=E_{0}$ using the Parisi formula, whose solution in the
pure case has 1 step of replica symmetry breaking. Another way to
prove the matching bound $\Es\geq E_{0}$ is to show that ${\rm Crt}_{N}((-\infty,u))$
concentrates around its mean for levels $u$ in some right neighborhood of
$-E_{0}$. This was proved in \cite{2nd} by a second moment argument,
independently of the Parisi formula.

Our main result is the following asymptotic matching of the second
moment and first moment squared of $\mbox{Crt}_{N}((-\infty,u))$,
for large enough $p$ and all energies $u>-E_{\infty}$.
\begin{thm}
\label{thm:main}For any $p\geq32$ and $u\in(-E_{\infty},\,\infty)$,
\begin{equation}
\lim_{N\to\infty}\frac{\mathbb{E}\big\{\left({\rm Crt}_{N}\left(\left(-\infty,u\right)\right)\right)^{2}\big\}}{\big(\mathbb{E}\left\{ {\rm Crt}_{N}\left(\left(-\infty,u\right)\right)\right\} \big)^{2}}=1.\label{eq:moment_matching}
\end{equation}
Consequently, in $L^{2}$ and in probability,
\begin{equation}
  \label{eq-120521a}
\frac{{\rm Crt}_{N}\left(\left(-\infty,u\right)\right)}{\E{\rm Crt}_{N}\left(\left(-\infty,u\right)\right)} \overset{N\to\infty}{\longrightarrow} 1.
\end{equation}
\end{thm}

The same result was proved in \cite{2nd} for any $u\in(-\Es,-E_{\infty})$
and general $p\geq3$. Moreover, it was proved there that for any
$u\in(-\Es,\,\infty)$ and $p\geq3$, the moments match at exponential
scale, i.e., $\frac{1}{N}\log$ of the ratio in (\ref{eq:moment_matching})
goes to $0$, a fact that will be used in our proof of Theorem \ref{thm:main}.
The value $-E_\infty$ represents a threshold beyond which a certain
shift
(corresponding to $U$,
see \eqref{eq:1stmom-1}) enters the bulk of the spectrum of a certain
GOE matrix $X_{N-1}$ related to the Hessian of $H_N (\bs)$; compared with \cite{2nd},
the latter fact requires a more delicate analysis of
the asymptotics of certain determinants appearing in the computation of
second moments, see \eqref{eq:2ndmom_refinedform-1}. The restriction to
$p\geq 32$ is needed in our analysis of the latter asymptotics. We expect that Theorem \ref{thm:main} remains true for $p\geq 3$, but we do not have a
proof for that; while our methods probably can be refined to yield a lower threshold than $32$, it is clear that additional ingredients are needed to
lower the threshold  to $3$. In particular, the use of H\"{o}lder's inequality in one step
of the proof, see \eqref{eq-Holder},
would need to be replaced by directly evaluating determinants
of correlated matrices, on a rare event.

Combining  \eqref{eq-120521a} with
\cite[Corollary 2]{2nd} and \eqref{eq:complexity}, we obtain the following.

\begin{cor}
	\label{cor:main}
	For any $p\geq32$ and $u\in(-\Es,\,\infty)$,
	\begin{equation*}
	\frac{1}{N}\log {\rm Crt}_{N}\left(\left(-\infty,u\right)\right) \overset{N\to\infty}{\longrightarrow} \Theta_p(u), \qquad \mbox{in probability.}
	\end{equation*}
\end{cor}

Before we turn to the proofs, we mention earlier related works concerning the critical points of  spin glass models. We start with the \emph{pure} spherical models which we consider in the current work.
Their mean complexity was first studied in physics in the context of the TAP approach, see for example the works of Crisanti and Sommers \cite{CrisantiSommersTAPpspin} and Crisnati, Leuzzi and Rizzo \cite{TAP-pSPSG4}. Fyodorov \cite{Fyodorov2004} and Fyodorov and Williams \cite{FyodorovWilliams} also studied the mean complexity of a closely related model, using the Kac-Rice formula.
In the rigororus
mathematical literature, the complexity was first studied in the seminal work of Auffinger, Ben Arous, and {\v{C}}ern{\'y} \cite{A-BA-C}. As mentioned above, similarly to \eqref{eq:complexity}, they also computed the mean complexity of critical points of a given index.
Auffinger and Gold \cite{AuffingerGold} recently claimed
that the second moment and first moment squared of the complexity of critical points of any fixed index match at exponential scale, for the relevant energies below $-E_\infty$. In \cite{pspinext} the authors studied the extremal point process formed by critical values in the vicinity of the global minimum of $H_N(\bs)$ and showed that, when properly centered, it converges to a Poisson process of exponential intensity. The connection between the same critical points and the Gibbs measure at sufficiently low temperature was studied in \cite{geometryGibbs}. In particular, it was shown there that the Gibss measure concentrates on spherical bands of temperature dependent radius around few of the deepest critical points.
The relevance of critical points with values macroscopically larger than the ground state energy to the free energy, at temperatures near the critical temperature,
was
recently studied by Ben Arous and Jagannath  \cite{BenArousJagannathShattering}.

Moving to the \emph{mixed} spherical models, Auffinger and Ben Arous \cite{ABA2} computed the mean complexity for general mixtures. As in \cite{A-BA-C}, their results also cover the mean complexity of critical points of a given index.
Matching of the second moment and first moment squared of the complexity at exponential scale was proved by the authors and Ben Arous in \cite{geometryMixed} for mixed models sufficiently close to a pure model in an appropriate sense, and energies close to the ground state. For the same models and at low temperature, some analogues of the results on the Gibbs measure from \cite{geometryGibbs} were obtained in \cite{geometryMixed}. In contrast to the pure models, for the mixed models the Gibbs measure was shown there to concentrate on bands around points which approximately maximize the restriction of $H_N(\bs)$ to a sphere of radius $\sqrt{Nq}$, for appropriate $q\in(0,1)$. For general mixed spherical models at arbitrary sub-critical temperature the same was proved in \cite{FElandscape}. In fact, when adding an appropriate TAP correction, the same even holds for general mixed models with Ising spins, see the works of Chen, Panchenko and
the first author \cite{TAPChenPanchenkoSubag,TAPIIChenPanchenkoSubag}.

For the \emph{bipartite} spherical mixed models, bounds on the mean complexity were computed by Auffinger and Chen \cite{AuffingerChenBipartite}. Recently, McKenna \cite{McKennaBipartite} computed the exact asymptotics of the complexity, building on a more general framework  developed by him, Ben Arous and Bourgade in \cite{BA-B-MK2,BA-B-MK1}. Finally, Kivimae \cite{KivimaeBipartite} proved the asymptotic matching of second moment and first moment squared of the complexity at exponential scale for the pure bipartite models with large enough number of spins and at energies sufficiently close to the ground state.

%
%

\section{Proofs}

The overlap of two points $\bs$ and $\bs'$ is defined as $R(\bs,\bs'):=\frac{1}{N}\bs\cdot\bs'$.
For any subset $I_{R}\subset\left[-1,1\right]$, we define

\[
\begin{aligned}\left[{\rm Crt}_{N}\left(B,I_{R}\right)\right]_{2}:=\# & \Big\{\left(\boldsymbol{\sigma},\boldsymbol{\sigma}'\right)\in\left(\SN\right)^{2}:\,R\left(\boldsymbol{\sigma},\boldsymbol{\sigma}'\right)\in I_{R},\\
	& \qquad \nabla H_{N}\left(\boldsymbol{\sigma}\right)=\nabla H_{N}\left(\boldsymbol{\sigma}'\right)=0,\,H_{N}\left(\boldsymbol{\sigma}\right),\,H_{N}\left(\boldsymbol{\sigma}'\right)\in NB\Big\}.
\end{aligned}
\]
That is, $\left[\mbox{Crt}_{N}\left(B,I_{R}\right)\right]_{2}$ is
the number of pairs of critical points $(\boldsymbol{\sigma},\boldsymbol{\sigma}')$
with overlap in $I_{R}$ and values in $NB$. Note that $\mathbb{E}\left[\mbox{Crt}_{N}\left(B,[-1,1]\right)\right]_{2}$
is the second moment of $\mbox{Crt}_{N}\left(B\right)$.

Our proof starts from the integral formulas for $\E{\rm Crt}_{N}\left(B\right)$
and $\E\left[\mbox{Crt}_{N}\left(B,I_{R}\right)\right]_{2}$ derived
in \cite{A-BA-C} and \cite{2nd}, using the Kac-Rice formula, which
we recall in Section \ref{subsec:Kac-Rice}. When we substitute the
latter formulas to (\ref{eq:moment_matching}), the proof of Theorem
\ref{thm:main} reduces to establishing certain upper bounds for large
$N$ on the corresponding integrands, which depend on an overlap value
in $[-1,1]$. These bounds will be proved in Sections \ref{subsec:large_overlaps}-\ref{subsec:small_overlaps}
for different ranges of the overlap. In Section \ref{subsec:Combining-the-bounds}
we will combine those bounds and complete the proof of Theorem \ref{thm:main}.
Lastly, in Section \ref{sec:pfofcor1} we will prove Corollary \ref{cor:main}.

\subsection{\label{subsec:Kac-Rice}Kac-Rice formulas }

The following representation for the expectation of $\mbox{Crt}_{N}\left(B\right)$
was proved in \cite{A-BA-C}, using the Kac-Rice formula (for the latter,
see e.g.
\cite[Theorem 12.1.1]{RFG}). Throughout the paper, the normalization
we use for a GOE matrix of dimension $N\times N$ is that the variance
is $1/N$ off-diagonal and $2/N$ on-diagonal.  We denote by $I$ the identity matrix.
\begin{lem}
\cite[Lemmas 3.1, 3.2]{A-BA-C}\label{lem:KR1stmoment} For any $p\geq3$
and any interval $B\subset\R$,
\begin{equation}
\mathbb{E}\left\{ {\rm Crt}_{N}\left(B\right)\right\} =\omega_{N}\left(\frac{p-1}{2\pi}(N-1)\right)^{\frac{N-1}{2}}\mathbb{E}\left\{ \left|\det\left(\mathbf{X}_{N-1}-\sqrt{\frac{1}{N-1}\frac{p}{p-1}}UI\right)\right|\mathbf{1}\Big\{ U\in\sqrt{N}B\Big\}\right\} ,\label{eq:1stmom-1}
\end{equation}
where $\mathbf{X}_{N-1}$ is a GOE matrix of dimension $N-1\times N-1$
independent of $U\sim N\left(0,1\right)$, and
\[
\omega_{N}:=\frac{2\pi^{N/2}}{\Gamma\left(N/2\right)}
\]
is the surface area of the $N-1$-dimensional unit sphere.
\end{lem}

In Section \ref{sec:Conditional-distributions} we will define the
two pairs $\big(\mathcal{\mathbf{X}}_{N-1}^{\left(i\right)}(r)\big)_{i=1}^{2}$
and $\big(\mathcal{\mathbf{E}}_{N-1}^{\left(i\right)}(r,u_{1},u_{2})\big)_{i=1}^{2}$
of $N-1\times N-1$ matrices, for any $r\in(-1,1)$ and $u_{1},\,u_{2}\in\R$.
The random matrices $\mathcal{\mathbf{X}}_{N-1}^{\left(i\right)}(r)$
are correlated, and the upper-left $N-2\times N-2$ submatrix of each
of them has the same distribution as the corresponding submatrix of a GOE matrix. The deterministic
matrices $\mathcal{\mathbf{E}}_{N-1}^{\left(i\right)}(r,u_{1},u_{2})$
are zero everywhere except for
the one element in the right-bottom corner. Let $\left(U_{1}\left(r\right),U_{2}\left(r\right)\right)\sim N\left(0,\Sigma_{U}\left(r\right)\right)$
be two Gaussian variables with covariance matrix $\Sigma_{U}\left(r\right)$
given by (\ref{eq:26}), independent of the latter matrices. The following
representation for the expectation of $\left[\mbox{Crt}_{N}\left(B,I_{R}\right)\right]_{2}$
was proved in \cite{2nd}.
\begin{lem}
\cite[Lemma 4]{2nd}\label{lem:KR2ndmoment} For any $p\geq3$ and
any intervals $B\subset\R$ and $I_{R}\subset\left(-1,1\right)$,
\begin{align}
 \label{eq:2ndmom_refinedform-1}
 & \mathbb{E}\left\{ \left[{\rm Crt}_{N}\left(B,I_{R}\right)\right]_{2}\right\} =C_{N}\int_{I_{R}}dr\left(\mathcal{G}\left(r\right)\right)^{N}\mathcal{F}\left(r\right) \\
 & \times\mathbb{E}\left\{ \prod_{i=1,2}\left|\det\left(\mathbf{X}_{N-1}^{\left(i\right)}(r)-\sqrt{\frac{1}{N-1}\frac{p}{p-1}}U_{i}\left(r\right)I+\mathbf{E}_{N-1}^{\left(i\right)}(r,U_{1}\left(r\right),U_{2}\left(r\right))\right)\right|\cdot\mathbf{1}\Big\{ U_{1}\left(r\right),U_{2}\left(r\right)\in\sqrt{N}B\Big\}\right\} ,\nonumber
\end{align}
where
\begin{align}
C_{N} & =\omega_{N}\omega_{N-1}\left(\frac{\left(N-1\right)\left(p-1\right)}{2\pi}\right)^{N-1},\,\,\,\mathcal{G}\left(r\right)=\left(\frac{1-r^{2}}{1-r^{2p-2}}\right)^{\frac{1}{2}},\label{eq:cgf-1}\\
\mathcal{F}\left(r\right) & =\left(\mathcal{G}\left(r\right)\right)^{-3}\left(1-r^{2p-2}\right)^{-\frac{1}{2}}\left(1-\left(pr^{p}-\left(p-1\right)r^{p-2}\right)^{2}\right)^{-\frac{1}{2}}.\nonumber
\end{align}
\end{lem}
For small overlaps $r$, we have the following.
\begin{cor}
\label{cor:smalloverlaps} For any $p\geq3$ and constant $T>0$,
there exists a function $B(r)=O(r^{3})$, which may depend on $T$, such that
the following holds. For any $a_{N},\,b_{N}\in[-T,T]$  and interval
$I_{N}=(r_{N},r_{N}')$ such that $r_{N},\,r_{N}'\to0$,
\begin{equation}
\limsup_{N\to\infty}\frac{\mathbb{E}\big\{\left[{\rm Crt}_{N}\left(\left(a_{N},b_{N}\right),I_{N}\right)\right]_{2}\big\}}{\big(\mathbb{E}\left\{ {\rm Crt}_{N}\left(\left(a_{N},b_{N}\right)\right)\right\} \big)^{2}}\leq \limsup_{N\to\infty}\sup_{u_1,u_2\in (a_{N},b_{N}) }\int_{I_N}\sqrt{\frac{N}{2\pi}}e^{-\frac{1}{2}Nr^{2}+NB(r)}\Delta_{N}(r,u_{1},u_{2})dr,\label{eq:smalloverlapsKR}
\end{equation}
where
\begin{equation}
  \Delta_{N}(r,u_{1},u_{2}):=\frac{\mathbb{E}\left\{ \prod_{i=1,2}\left|\det\left(\mathbf{X}_{N-1}^{\left(i\right)}(r)-\sqrt{\frac{N}{N-1}\frac{p}{p-1}}u_{i}I+\mathbf{E}_{N-1}^{\left(i\right)}(r,\sqrt{N}u_{1},\sqrt{N}u_{2})\right)\right|\right\} }{\prod_{i=1,2}\mathbb{E}\left\{ \left|\det\left(\mathbf{X}_{N-1}-\sqrt{\frac{N}{N-1}\frac{p}{p-1}}u_{i}I\right)\right|\right\} },\label{eq:DeltaN}
\end{equation}
and $\mathbf{X}_{N-1}$ is distributed as in Lemma \ref{lem:KR1stmoment}.
\end{cor}

\begin{proof}
The corollary follows from Lemmas \ref{lem:KR1stmoment} and \ref{lem:KR2ndmoment}
since as $r\to0$, $\mathcal{F}\left(r\right)=1+O(r)$, $\mathcal{G}\left(r\right)=e^{-\frac{1}{2}r^{2}+O(r^{4})}$
and uniformly in $r\in I_{N}$ and $u_{1},\,u_{2}\in(a_{N},b_{N})$,
\begin{align*}
\varphi_{\Sigma_{U}\left(r\right)}(u_{1},u_{2}) & :=\frac{1}{2\pi}\left(\det\left(\Sigma_{U}\left(r\right)\right)\right)^{-1/2}\exp\left\{ -\frac{1}{2}\left(u_{1},u_{2}\right)\left(\Sigma_{U}\left(r\right)\right)^{-1}\left(u_{1},u_{2}\right)^{T}\right\} \\
 & =\left(1+O\left(r^{p}\right)\right)\frac{1}{2\pi}\exp\left\{ -\frac{1}{2}\left(u_{1}^{2}+u_{2}^{2}\right)+\left(u_{1}+u_{2}\right)^{2}O(r^{p})\right\} ,
\end{align*}
and since
\[
\lim_{N\to\infty}\frac{C_{N}}{\omega_{N}^{2}\left(\frac{p-1}{2\pi}(N-1)\right)^{N-1}}/\sqrt{\frac{N}{2\pi}}=\lim_{N\to\infty}\frac{\omega_{N-1}}{\omega_{N}}/\sqrt{\frac{N}{2\pi}}=1.
\]
\end{proof}

\subsection{\label{subsec:large_overlaps}Large overlaps and energies}

The following lemma follows from the asymptotics at exponential scale
derived in \cite{A-BA-C} and \cite{2nd}.
\begin{lem}
\label{lem:large}For any $p\geq3$ and $u\in(-E_\infty,\,\infty)$, there
exist some sequences of real numbers $a_{N},\,b_{N}\to\min\{u,0\}$
and $\rho_{N}\to0$, such that
\[
\lim_{N\to\infty}\frac{\mathbb{E}\big\{\left({\rm Crt}_{N}\left(\left(-\infty,u\right)\right)\right)^{2}\big\}}{\big(\mathbb{E}\left\{ {\rm Crt}_{N}\left(\left(-\infty,u\right)\right)\right\} \big)^{2}}/\frac{\mathbb{E}\big\{\left[{\rm Crt}_{N}\left(\left(a_{N},b_{N}\right),(-\rho_{N},\rho_{N})\right)\right]_{2}\big\}}{\big(\mathbb{E}\left\{ {\rm Crt}_{N}\left(\left(a_{N},b_{N}\right)\right)\right\} \big)^{2}}=1.
\]
\end{lem}

\begin{proof}
Let $\rho\in\left(0,1\right)$ and $\epsilon>0$. To prove the lemma
it is enough to show that if $u<0$ then
\begin{align}
\label{eq:ratio1}\lim_{N\to\infty}\mathbb{E}\left\{ \left[{\rm Crt}_{N}\left(\left(u-\epsilon,u\right),(-\rho,\rho)\right)\right]_{2}\right\} /\mathbb{E}\left\{ \left({\rm Crt}_{N}\left(\left(-\infty,u\right)\right)\right)^{2}\right\}  & =1,\\
\label{eq:ratio2}\lim_{N\to\infty}\mathbb{E}\left\{ {\rm Crt}_{N}\left(\left(u-\epsilon,u\right)\right)\right\} /\mathbb{E}\left\{ {\rm Crt}_{N}\left(\left(-\infty,u\right)\right)\right\}  & =1,
\end{align}
and if $u\geq0$ then
\begin{align}
\label{eq:ratio3}\lim_{N\to\infty}\mathbb{E}\left\{ \left[{\rm Crt}_{N}\left(\left(-\epsilon,\epsilon\right),(-\rho,\rho)\right)\right]_{2}\right\} /\mathbb{E}\left\{ \left({\rm Crt}_{N}\left(\left(-\infty,u\right)\right)\right)^{2}\right\}  & =1,\\
\label{eq:ratio4}\lim_{N\to\infty}\mathbb{E}\left\{ {\rm Crt}_{N}\left(\left(-\epsilon,\epsilon\right)\right)\right\} /\mathbb{E}\left\{ {\rm Crt}_{N}\left(\left(-\infty,u\right)\right)\right\}  & =1.
\end{align}

The equalities  \eqref{eq:ratio2} and \eqref{eq:ratio4} follow from Theorem 2.8 of \cite{A-BA-C}.
For $u\in(-\Es,\,-E_{\infty})$, \eqref{eq:ratio1} was proved in Lemma 20 of \cite{2nd}.
The proof of \eqref{eq:ratio1} and \eqref{eq:ratio3}
(for $u>-E_\infty$)
follows by the same argument
if one is able to extend Part (i) of Lemma 6 of \cite{2nd},
used in the proof of Lemma 20 there, to subsets $B\subset (-E_\infty,E_\infty) $ instead of
subsets $B\subset (-\infty,-E_{\infty})$ as in the original statement of the lemma.
The proof of the
latter generalizes easily if one shows that for any $r\in(-1,1)$,
$g_{r}(u_{1},u_{2})$ is strictly concave on $(-E_\infty,E_\infty)^2$  where
\begin{align*}
g_{r}(u_{1},u_{2}) & =-\frac{1}{2}(u_{1},u_{2})\big(\Sigma_{U}(r)\big)^{-1}\left(\begin{array}{c}
u_{1}\\
u_{2}
\end{array}\right)
  +\Omega\Big(\sqrt{\frac{p}{p-1}}u_{1}\Big)+\Omega\Big(\sqrt{\frac{p}{p-1}}u_{2}\Big)
\end{align*}
and
\[
\Omega(x)=\begin{cases}
\frac{x^{2}}{4}-\frac{1}{2} & \text{if }|x|\leq2\\
\frac{x^{2}}{4}-\frac{1}{2}-\frac{|x|}{2}\sqrt{\frac{x^{2}}{4}-1}+\log\bigg(\sqrt{\frac{x^{2}}{4}-1}+\frac{|x|}{2}\bigg) & \text{if }|x|>2.
\end{cases}
\]
The concavity of $g_{r}(u_{1},u_{2})$ was proved in \cite{2nd} for $(u_{1},u_2)\in (-\infty,-E_\infty)^2$, we complete the current
proof by proving concavity for  $(u_{1},u_2)\in (-E_\infty,E_\infty)^2$.
For  such $(u_{1},u_2)\in (-E_\infty,E_\infty)^2$, the Hessian of $g_{r}(u_{1},u_{2})$ equals
\[
-\big(\Sigma_{U}(r)\big)^{-1}
+\left(\begin{array}{cc}
	\frac{1}{2}\frac{p}{p-1}  & 0 \\
	0 & \frac{1}{2}\frac{p}{p-1}
\end{array}\right),
\]
since for $x\in (-E_\infty,E_\infty)$,
\[
\frac{d^2}{dx^2}\Big(\Omega\Big(\sqrt{\frac{p}{p-1}}x\Big)\Big)=\frac{1}{2}\frac{p}{p-1}.
\]


Hence, if we show that the eigenvalues\footnote{The eigenvalues of $\Sigma_{U}(r)$ are positive for any $r\in(-1,1)$,
see Remark 31 in \cite{2nd}.} of $\Sigma_{U}(r)$ are less than $2(p-1)/p$, it will follow that
$g_{r}(u_{1},u_{2})$ is strictly concave on $(-E_\infty,E_\infty)^2$.

The eigenvalues of $\Sigma_{U}(r)$ are equal to $\Sigma_{U,11}(r)\pm\Sigma_{U,12}(r)$.
One can verify that
\[
\Sigma_{U,11}(r)\pm\Sigma_{U,12}(r)=\frac{1-r^{2p-2}\pm(p-1)r^{p-2}(1-r^{2})}{1\mp(pr^{p}-(p-1)r^{p-2})},
\]
which are less than $2(p-1)/p$ if
\begin{equation}
  h(r):=p-2+pr^{2p-2}\pm(p-2)(p-1)r^{p-2}\mp(p-1)pr^{p}>0.\label{eq:hr}
\end{equation}
Since $h(0)>0$ and $h(\pm1)\geq0$, to prove that $h(r)>0$ for $r\in(-1,1)$
it will be enough to show the same for any critical point $r$ of
$h(r)$. Since
\[
h'(r)=\frac{1}{r}(2p-2)\left[pr^{2p-2}\pm\frac{1}{2}(p-2)^{2}r^{p-2}\mp\frac{1}{2}p^{2}r^{p}\right],
\]
at a critical point
\[
pr^{2p-2}=\mp\frac{1}{2}(p-2)^{2}r^{p-2}\pm\frac{1}{2}p^{2}r^{p}
\]
and
\[
h(r)=p-2\pm\frac{1}{2}p(p-2)r^{p-2}(1-r^{2}).
\]
Since $|r^{p-2}(1-r^{2})|<2/p$ for all $r\in(-1,1)$, we conclude
(\ref{eq:hr}) and that $g_{r}(u_{1},u_{2})$ is strictly concave on $(-E_\infty,E_\infty)^2$.
\end{proof}

\subsection{\label{subsec:intermediate_overlaps}Intermediate overlaps}

Lemma \ref{lem:large} allows us to restrict to small overlap values
$r\in(-\rho_{N},\rho_{N})$, but with no specific bound on the rate at which
$\rho_{N}\to0$.
In this section we prove that for appropriate $C>0$, overlaps $r$
with $|r|\geq C\sqrt{\log N/N}$ are negligible as $N\to\infty$.
\begin{lem}
\label{lem:Intermediate}For any $p\geq3$ and constant $T\in (0,E_\infty)$, the
following holds, with the function $B(r)=O(r^{3})$
from Corollary \ref{cor:smalloverlaps}. For any sequence $\rho_{N}\to0$,
 setting
\[
I_{N}:=(-\rho_{N},\rho_{N})\setminus(-C\sqrt{\log N/N},C\sqrt{\log N/N}),
\]
for some $C>\sqrt{2}$, we have that
\begin{equation}
\limsup_{N\to\infty}\sup_{u_1,u_2\in [-T,T] }\int_{I_N}\sqrt{\frac{N}{2\pi}}e^{-\frac{1}{2}Nr^{2}+NB(r)}\Delta_{N}(r,u_{1},u_{2})dr=0.
\label{eq:large_r}
\end{equation}
\end{lem}

\begin{proof}
Denote by $\mathbf{M}^{(i)}=\mathbf{M}_{N-1}^{(i)}(r,u_{1},u_{2})$
and $\mathbf{W}^{(i)}=\mathbf{W}_{N-1}(u_{i})$ the matrices in the
numerator and denominator of (\ref{eq:DeltaN}), respectively. By
the Cauchy-Schwarz inequality,
\[
\Delta_{N}(r,u_{1},u_{2})\leq\frac{\sqrt{\mathbb{E}\left\{ \det\left(\mathbf{M}^{(1)}\right)^{2}\right\} \mathbb{E}\left\{ \det\left(\mathbf{M}^{(2)}\right)^{2}\right\} }}{\mathbb{E}\left\{ \left|\det\left(\mathbf{W}^{(1)}\right)\right|\right\} \mathbb{E}\left\{ \left|\det\left(\mathbf{W}^{(2)}\right)\right|\right\} }.
\]
Hence, to prove the lemma it will be enough to show that for large
$N$, uniformly in $r\in I_{N}$ and $u_{1},\,u_{2}\in[-T,T]$,
\begin{equation}
\frac{\mathbb{E}\left\{ \det\left(\mathbf{M}^{(i)}\right)^{2}\right\} }{\mathbb{E}\left\{ \left|\det\left(\mathbf{W}^{(i)}\right)\right|\right\} ^{2}}\leq cN,\label{eq:det2moment}
\end{equation}
 for some constant $c$, since this would imply that
\begin{align*}
 & \int_{I_{N}}\sqrt{\frac{N}{2\pi}}e^{-\frac{1}{2}Nr^{2}+NB(r)}\Delta_{N}(r,u_{1},u_{2})dr\\
 & \leq cNe^{o_{N}(1)}\int_{\sqrt{N}I_{N}}\frac{1}{\sqrt{2\pi}}e^{-\frac{1}{2}r^{2}}dr\leq\sqrt{\frac{2}{\pi\log N}}\frac{c}{C}e^{(1-\frac{1}{2}C^{2})\log N+o_{N}(1)},
\end{align*}
where we used the fact that for the standard Gaussian density $\varphi(x)$,
for $x\geq0$, $\int_{x}^{\infty}\varphi(t)dt<\varphi(x)/x$. We prove
a slightly more general result than (\ref{eq:det2moment}) in the
lemma below, which will be also used in Section \ref{subsec:small_overlaps}.
\end{proof}
\begin{lem}
\label{lem:detMomentBd} Let
$\mathbf{M}^{(i)}=\mathbf{M}_{N-1}^{(i)}(r,u_{1},u_{2})$
and $\mathbf{W}^{(i)}=\mathbf{W}_{N-1}(u_{i})$ denote the matrices in the
numerator and denominator of (\ref{eq:DeltaN}), respectively. Then
for any even integer $k>1$ and $T\in (0,E_\infty)$, there exists a constant $c=c(k,T)>0$
such that for large $N$,
\[
\frac{\mathbb{E}\left\{ \left|\det\mathbf{M}^{(i)}\right|^{k}\right\} }{\mathbb{E}\left\{ \left|\det\mathbf{W}^{(i)}\right|\right\} ^{k}}\leq cN^{\frac{k(k-1)}{2}},
\]
uniformly in $r\in(-1,1)$ and $u_{1},\,u_{2}\in[-T,T]$.
\end{lem}

\begin{proof}
Using (\ref{eq:ghat-1}),
\begin{align*}
\mathbf{M}^{(i)} & =\left(\begin{array}{cc}
\mathbf{G}_{N-2}^{\left(i\right)}\left(r\right)-\bar{u}_{i}I & Z^{\left(i\right)}\left(r\right)\\
\left(Z^{\left(i\right)}\left(r\right)\right)^{T} & Q^{\left(i\right)}\left(r\right)-\bar{u}_{i}+\bar{m}_{i}
\end{array}\right),
\end{align*}
where $\bar{u}_{i}:=\sqrt{\frac{N}{N-1}\frac{p}{p-1}}u_{i}$ and $\bar{m}_{i}:=\sqrt{\frac{N}{(N-1)p(p-1)}}m_{i}\left(r,u_{1},u_{2}\right)$ and where
$m_{i}\left(r,u_{1},u_{2}\right)$,
$\mathbf{G}_{N-2}^{\left(i\right)}\left(r\right)$,
$Z^{\left(i\right)}\left(r\right)$
and
$Q^{\left(i\right)}\left(r\right)$   are defined in Section \ref{sec:Conditional-distributions}.
To lighten the notation, we henceforth omit $i$ from the super- and
sub-scripts for all variables.

Let $\tilde{A}_{N-1}$ be an orthogonal matrix (measurable w.r.t.
 $Z\left(r\right)$) such that
\[
\tilde{A}_{N-1}Z\left(r\right)=\tilde{Z}\left(r\right):=(0,\ldots,0,\|Z\left(r\right)\|),
\]
and set
\[
A_{N-1}=\left(\begin{array}{cc}
\tilde{A}_{N-1} & 0\\
0 & 1
\end{array}\right).
\]
Then, denoting $\tilde{\mathbf{G}}_{N-2}\left(r\right):=\tilde{A}_{N-1}\mathbf{G}_{N-2}\left(r\right)\tilde{A}_{N-1}^{T}$,
we have that
\[
A_{N-1}\mathbf{M}A_{N-1}^{T}=\left(\begin{array}{cc}
\tilde{\mathbf{G}}_{N-2}\left(r\right)-\bar{u}I & \tilde{Z}\left(r\right)\\
\tilde{Z}\left(r\right)^{T} & Q\left(r\right)-\bar{u}+\bar{m}
\end{array}\right).
\]

Therefore,
\begin{align*}
\det\mathbf{M}=\det\left(A_{N-1}\mathbf{M}A_{N-1}^{T}\right) & =(Q\left(r\right)-\bar{u}+\bar{m})\det\left(\tilde{\mathbf{G}}_{N-2}\left(r\right)-\bar{u}I\right)\\
 & -\|Z\left(r\right)\|^{2}\det\left(\tilde{\mathbf{G}}_{N-3}\left(r\right)-\bar{u}I\right),
\end{align*}
where by an abuse of notation we denote by $\tilde{\mathbf{G}}_{N-3}\left(r\right)$
the $N-3\times N-3$ upper-left submatrix of $\tilde{\mathbf{G}}_{N-2}\left(r\right)$.

Note that $\tilde{\mathbf{G}}_{N-2}\left(r\right)$, $\tilde{Z}\left(r\right)$
and $Q\left(r\right)$ are independent, and that $\tilde{\mathbf{G}}_{N-2}\left(r\right)$
has the same law as $\mathbf{G}_{N-2}\left(r\right)$. Specifically,
if we let $\hat{\mathbf{G}}_{j}$ be a $j\times j$ centered, symmetric
matrix with independent elements, up to symmetry, whose off-diagonal
elements have variance $1/(N-1)$ and whose on-diagonal elements have
variance $2/(N-1)$, then, for any fixed $r$, $\tilde{\mathbf{G}}_{N-2}\left(r\right)$
and $\tilde{\mathbf{G}}_{N-3}\left(r\right)$ have the same law as
$\hat{\mathbf{G}}_{N-2}$ and $\hat{\mathbf{G}}_{N-3}$. Using the
fact that $|x+y|^{k}\leq2^{k-1}(|x|^{k}+|y|^{k})$,
\begin{align*}
\E\Big\{|\det\mathbf{M}|^{k}\Big\} & \leq2^{k-1}\bigg[\E\Big(|Q\left(r\right)-\bar{u}+\bar{m}|^{k}\Big)\E\Big(\Big|\det\Big(\hat{\mathbf{G}}_{N-2}-\bar{u}I\Big)\Big|^{k}\Big)\\
 & +\E\Big(\|Z\left(r\right)\|^{2k}\Big)\E\Big(\Big|\det\Big(\hat{\mathbf{G}}_{N-3}-\bar{u}I\Big)\Big|^{k}\Big)\bigg].
\end{align*}

Since $\mathbf{W}$ has the same law as $\hat{\mathbf{G}}_{N-1}-\bar{u}I$,
to prove the lemma we need to bound the ratio of the above and
\[
\E\Big\{|\det\mathbf{W}|\Big\}^{k}=\E\Big\{\Big|\det\Big(\hat{\mathbf{G}}_{N-1}-\bar{u}I\Big)\Big|\Big\}^{k}.
\]

One can check that for large $N$,  $r\in(-1,1)$ and $u_i\in(-E_\infty,E_\infty)$, the variance of each
element of $Z\left(r\right)$ is bounded by $\frac{1}{N-1}$ and
\begin{align*}
\E\Big(|Q\left(r\right)-\bar{u}+\bar{m}|^{k}\Big) & \leq3^{k}\\
\E\Big(\|Z\left(r\right)\|^{2k}\Big) & \leq2.
\end{align*}
Hence, it remains to show that for large $N$ and $k>1$ even,
\begin{equation}
  \label{eq-detcomp}
\frac{\E\Big(\Big|\det\Big(\hat{\mathbf{G}}_{N-2}-\bar{u}I\Big)\Big|^{k}\Big)}{\E\Big\{\Big|\det\Big(\hat{\mathbf{G}}_{N-1}-\bar{u}I\Big)\Big|\Big\}^{k}},\,\frac{\E\Big(\Big|\det\Big(\hat{\mathbf{G}}_{N-3}-\bar{u}I\Big)\Big|^{k}\Big)}{\E\Big\{\Big|\det\Big(\hat{\mathbf{G}}_{N-1}-\bar{u}I\Big)\Big|\Big\}^{k}}\leq CN^{\frac{k(k-1)}{2}}
\end{equation}
for some constant $C>0$, uniformly in $\bar{u}\in[-\bar T,\bar T]$ for $\bar T:= T  \sqrt{\frac{N}{N-1}\frac{p}{p-1}}$. Note that by our assumption that $T\in(0,E_\infty)$, $\bar T <2-\epsilon$ for some small $\epsilon>0$ and large $N$.

Toward showing \eqref{eq-detcomp}, it is worthwhile to switch notation,
in order
to better fit existing literature.
Let $\bY_{N-1}$ denote a GOE matrix, of dimension $N-1$,
normalized so that its spectrum is supported on
$[-\sqrt{2},\sqrt{2}]$ (this is the scaling that \cite{A-BA-C} use,
i.e. the variance of the off-diagonal elements is
$1/2N$).  From \cite[Lemma 3.3]{A-BA-C} (taking $t\to 0$ with $N$ fixed)
we obtain that
\begin{equation}
\label{eq-1}
\mathbb{E}\big(|\mbox{\rm det} \bY_{N-1}-xI|\big)=\frac{\sqrt{2}\Gamma(\frac{N}{2})\sqrt{N(N-1)}}{(N-1)^{N/2} } e^{(N-1) x^2/2} \rho_{N}(x),
\end{equation}
where $\rho_N(x)dx=\E L_N(dx)$ is the expected density of states, which converges to the semicircle law at $x$ uniformly in compact subsets of $(-\sqrt{2},\sqrt{2})$.
Applying the Stirling formula, we conclude that
\begin{equation}
\label{eq-1a}
\mathbb{E}\big(|\mbox{\rm det} \bY_{N-1}-xI|\big)=C_1(x) 2^{-N/2}\sqrt{N} e^{-N/2+(N-1) x^2/2}(1+o(1)),
\end{equation}
where $C_1(x)$ is explicit, continuous in $x$, strictly positive
on $(-\sqrt{2},\sqrt{2})$, and the error is uniform in compact
subsets of $(-\sqrt{2},\sqrt{2})$.

On the other hand, using \cite[(56)]{BrezinHikam} (correcting for the factor
$e^{Nk x^2}2^{-Nk/2}$, which is missing from
their formulation because they omitted the values of $f$ at the saddle point, see \cite[(50),(54)]{BrezinHikam} ), we have that
\begin{equation}
\label{eq-2}
\mathbb{E}\big(|\mbox{\rm det} \bY_{N-1}-xI|^{2k}\big)=C_{2k}(x) N^{2k^2}
2^{-Nk}e^{-Nk+Nk x^2}(1+o(1)),
\end{equation}
where $C_{2k}(x)$ is an explicit constant, continuous in $x$ and
independent of $N$, and strictly positive in $(-\sqrt{2},\sqrt{2})$.

All together, we obtain that, uniformly on compact subsets of
$(-\sqrt{2},\sqrt{2})$,
\begin{equation}
\label{eq-3}
\frac{\mathbb{E}\big(|\mbox{\rm det} \bY_{N-1}-xI|^{2k}\big)}{\Big(\mathbb{E}\big(|\mbox{\rm det} \bY_{N-1}-xI|\big)\Big)^{2k}}=c_{2k}(x) N^{2k^2-k}(1+o(1)),
\end{equation}
where $c_{2k}(x)$ is some positive explicit constant, continuous in $x$.

The claim \eqref{eq-detcomp}
follows from \eqref{eq-3} together with the fact that the ratio of the right
hand side of \eqref{eq-1} for $N-1$ and $N-\ell$ with $\ell$ fixed is bounded
independently of $N$.
\end{proof}

\subsection{\label{subsec:small_overlaps}Small overlaps }

In this section we treat the range of overlaps from which comes the
main contribution to the ratio of the moments. Namely, we will prove
the following lemma.
\begin{lem}
\label{lem:small_r}For any $p\geq32$ and constants $C>0$ and $T\in (0,E_\infty)$, the
following holds, with the function  $B(r)=O(r^{3})$
from Corollary \ref{cor:smalloverlaps}. Setting
\begin{equation}
I_{N}:=(-C\sqrt{\log N/N},C\sqrt{\log N/N}),\label{eq:INsmall}
\end{equation}
uniformly in $u_{1},\,u_{2}\in[-T,T]$, we have that
\[
\limsup_{N\to\infty}\int_{I_{N}}\sqrt{\frac{N}{2\pi}}e^{-\frac{1}{2}Nr^{2}+NB(r)}\Delta_{N}(r,u_{1},u_{2})dr\leq 1.
\]
\end{lem}

Before we turn to the proof, we represent the matrices in the definition
of $\Delta_{N}(r,u_{1},u_{2})$ (see (\ref{eq:DeltaN})) as linear
combinations of independent variables. Fix some $r$, $u_{1}$ and
$u_{2}$. Let $V$, $V^{(1)}$ and $V^{(2)}$ be i.i.d. Gaussian vectors
of length $N-2$ each with distribution
\[
N\Big(0,\frac{1}{N-1}I_{N-2}\Big),
\]
where $I_{N-2}$ is the identity matrix. Let $D$, $D^{(1)}$ and
$D^{(2)}$ be i.i.d. Gaussian variables each with distribution
\[
N\Big(0,\frac{2}{N-1}\Big).
\]
Recall the definition of $Z^{\left(i\right)}(r)$ and $Q^{\left(i\right)}\left(r\right)$
from Section \ref{sec:Conditional-distributions}.

We may define all the variables above on the same probability space
such that
\[
\begin{aligned}Z_{j}^{\left(i\right)}\left(r\right) & =\sqrt{\frac{\Sigma_{Z,11}\left(r\right)-\left|\Sigma_{Z,12}\left(r\right)\right|}{p(p-1)}}V^{(i)}+\left(\mbox{sgn}\left(\Sigma_{Z,12}\left(r\right)\right)\right)^{i}\sqrt{\frac{\left|\Sigma_{Z,12}\left(r\right)\right|}{p(p-1)}}V,\\
Q^{\left(i\right)}\left(r\right) & =\sqrt{\frac{\Sigma_{Q,11}\left(r\right)-\left|\Sigma_{Q,12}\left(r\right)\right|}{2p(p-1)}}D^{(i)}+\left(\mbox{sgn}\left(\Sigma_{Q,12}\left(r\right)\right)\right)^{i}\sqrt{\frac{\left|\Sigma_{Q,12}\left(r\right)\right|}{2p(p-1)}}D.
\end{aligned}
\]

Also recall the definition of $\mathbf{G}_{N-2}^{\left(i\right)}\left(r\right)$,
using the i.i.d. matrices $\bar{\mathbf{G}}_{N-2}$, $\bar{\mathbf{G}}_{N-2}^{\left(1\right)}$,
and $\bar{\mathcal{\mathbf{G}}}_{N-2}^{\left(2\right)}$, from Section
\ref{sec:Conditional-distributions}, and assume they are independent
of all other variables. Define
\begin{align*}
\hat{\mathbf{X}}_{N-1}^{\left(i\right)}\left(r\right) & =\left(\begin{array}{cc}
\mathbf{G}_{N-2}^{\left(i\right)}\left(r\right) & V^{\left(i\right)}\\
\left(V^{\left(i\right)}\right)^{T} & D^{\left(i\right)}
\end{array}\right),
\end{align*}
and note that for each $i=1,2$, $\hat{\mathbf{X}}_{N-1}^{\left(i\right)}\left(r\right)$
is a GOE matrix.

We have that
\[
\Delta_{N}(r,u_{1},u_{2})=\frac{\mathbb{E}\left\{ \prod_{i=1,2}\left|\det\left(\mathbf{A}_{N-1}^{\left(i\right)}+\mathbf{B}_{N-1}^{\left(i\right)}\right)\right|\right\} }{\prod_{i=1,2}\mathbb{E}\left\{ \left|\det\left(\mathbf{A}_{N-1}^{\left(i\right)}\right)\right|\right\} },
\]
where
\[
\mathbf{A}_{N-1}^{\left(i\right)}=\mathbf{A}_{N-1}^{\left(i\right)}(r,u_{i})=\hat{\mathbf{X}}_{N-1}^{\left(i\right)}\left(0\right)-\sqrt{\frac{N}{N-1}\frac{p}{p-1}}u_{i}I
\]
and
\begin{align*}
\mathbf{B}_{N-1}^{\left(i\right)} & =\mathbf{B}_{N-1}^{\left(i\right)}(r,u_{1},u_{2})=\mathbf{X}_{N-1}^{\left(i\right)}(r)-\hat{\mathbf{X}}_{N-1}^{\left(i\right)}\left(0\right)+\mathbf{E}_{N-1}^{\left(i\right)}(r,\sqrt{N}u_{1},\sqrt{N}u_{2})\\
 & =\left(\begin{array}{cc}
\mathbf{G}_{N-2}^{\left(i\right)}\left(r\right)-\mathbf{G}_{N-2}^{\left(i\right)}\left(0\right) & Z^{\left(i\right)}\left(r\right)-V^{\left(i\right)}\\
\left(Z^{\left(i\right)}\left(r\right)\right)^{T}-\left(V^{\left(i\right)}\right)^{T} & Q^{\left(i\right)}\left(r\right)-D^{\left(i\right)}+\frac{m_{i}\left(r,u_{1},u_{2}\right)}{\sqrt{(N-1)p(p-1)}}
\end{array}\right).
\end{align*}

To prove Lemma \ref{lem:small_r} we will need the following auxiliary
results which we prove below. For any $N\times N$ symmetric matrix $\mathbf{T}$
of denote by $\lambda_{i}(\mathbf{T})$, $i=1,\ldots,N$, the eigenvalues
of a matrix $\mathbf{T}$.
\begin{lem}{\bf (Bound on the perturbations $\mathbf{B}_{N-1}^{\left(i\right)}$)}
\label{lem:B_ub} Fix $T>0$.
For any $p\geq3$ and $K>0$ , there exists $c>0$
such that the following holds. If $t_{N}$ is a sequence such that
$\liminf_{N\to\infty}t_{N}(N/\log N)^{\frac{p-3}{4}}\geq c$, then
for large $N$, uniformly in $r\in I_{N}$ defined in (\ref{eq:INsmall})
and $u_1,u_2\in [-T,T]$,
\[
\P\left\{ \max_{j\leq N-1}|\lambda_{j}(\mathbf{B}_{N-1}^{\left(i\right)})|\geq t_{N}\right\} \leq e^{-KN}.
\]
\end{lem}

\begin{lem}{\bf (Overcrowding estimate)}
\label{lem:evgap}For an $N\times N$ GOE matrix $\mathbf{G}$, for
any measurable set $I\subset\R$,
\[
\P\left\{ \mathbf{G}\text{ has at least $t$ eigenvalues in }I\right\} \leq \frac{10N|I|}{t}.
\]
\end{lem}

Lastly, we will need the following corollary of the main theorem of
\cite{Miroslav}.
\begin{cor}
\cite{Miroslav}\label{cor:Miroslav-1} Let $\mathbf{C}_{1}$, $\mathbf{C}_{2}$
be two (deterministic) real, symmetric $N\times N$ matrices and let
$\lambda_{j}\left(\mathbf{C}_{i}\right)$ denote the eigenvalues of
$\mathbf{C}_{i}$, ordered with non-decreasing absolute value. Suppose
that the number of non-zero eigenvalues of $\mathbf{C}_{2}$ is $d$
at most. Then, assuming $\lambda_1(\mathbf{C}_{1})\neq0$,
\[
\left|\det\left(\mathbf{C}_{1}+\mathbf{C}_{2}\right)\right|\leq|\det(\mathbf{C}_{1})|\prod_{j=1}^{d}\left(1+\frac{\left|\lambda_{N}\left(\mathbf{C}_{2}\right)\right|}{\left|\lambda_{j}\left(\mathbf{C}_{1}\right)\right|}\right).
\]
\end{cor}

\subsubsection{Proof of Lemma \ref{lem:small_r}}

Since $B(r)=O(N^{-\frac{3}{2}+\epsilon})$ uniformly in $r\in I_{N}$ for arbitrarily
small $\epsilon>0$,
\[
\lim_{N\to\infty}\int_{I_{N}}\sqrt{\frac{N}{2\pi}}e^{-\frac{1}{2}Nr^{2}+NB(r)}dr=1.
\]
Hence, it will be enough to show that uniformly in $r\in I_{N}$ and
$u_{1},\,u_{2}\in[-T,T]$,
\begin{equation}
\limsup_{N\to\infty}\Delta_{N}(r,u_{1},u_{2})\leq1.\label{eq:ubDelta}
\end{equation}
Fix some $\alpha,\beta,\gamma>0$  (to be determined below)
and let $\mathcal{E}=\mathcal{E}_{N}(r,u_{1},u_{2})= \mathcal{E}_1\cap \mathcal{E}_2$ where
\[
\mathcal{E}_1=\bigcap_{i=1}^2 \bigcap_{\ell=0}^{\lceil\log_2N^{\alpha}\rceil}
\Big\{\#\left\{j:|\lambda_{j}(\mathbf{A}_{N-1}^{\left(i\right)})|\in [s_\ell,s_{\ell+1})\right\}\leq  N^{1+\beta}s_{\ell} \Big\},\quad
\mathcal{E}_2=\bigcap_{i=1}^2
\Big\{\max_{j}|\lambda_{j}(\mathbf{B}_{N-1}^{\left(i\right)})|\leq N^{-\gamma}\Big\},
\]
with
\[
s_\ell=
\begin{cases}
	0, &\ell=0\\
	\frac{2^\ell}{N^{\alpha}}, &1\leq \ell\leq \lceil\log_2N^{\alpha}\rceil\\
	\infty, &\ell=\lceil\log_2N^{\alpha}\rceil+1.
	\end{cases}
\]
From Corollary \ref{cor:Miroslav-1}, on the event $\mathcal{E}$,
\begin{align*}
\left|\det\left(\mathbf{A}_{N-1}^{\left(i\right)}+\mathbf{B}_{N-1}^{\left(i\right)}\right)\right|
&\leq\left|\det\left(\mathbf{A}_{N-1}^{\left(i\right)}\right)\right|
\prod_{\ell=1}^{\lceil\log_2N^{\alpha}\rceil}\left(1+\frac{N^{-\gamma}}{s_\ell}\right)^{ N^{1+\beta} s_{\ell}}\\
&\leq  \left|\det\left(\mathbf{A}_{N-1}^{\left(i\right)}\right)\right|  \exp\Big\{
 N^{1+\beta-\gamma} \lceil\log_2N^{\alpha}\rceil \Big\}
,
\end{align*}
where we used the fact that $(1+y)^x\leq e^{xy}$ for $x,y>0$.
Hence if
\begin{equation}
	\label{eq:cond1}\gamma-\beta>1,
\end{equation}
then
\[
\limsup_{N\to\infty}\frac{\mathbb{E}\left\{ \prod_{i=1,2}\left|\det\left(\mathbf{A}_{N-1}^{\left(i\right)}+\mathbf{B}_{N-1}^{\left(i\right)}\right)\right|\indic_{\mathcal{E}}\right\} }{\prod_{i=1,2}\mathbb{E}\left\{ \left|\det\left(\mathbf{A}_{N-1}^{\left(i\right)}\right)\right|\right\} }\leq1.
\]

For a given $p\geq3$, if we are able to choose $\alpha,\beta,\gamma>0$
such that \eqref{eq:cond1}
holds and further,
uniformly in $r\in I_{N}$ and $u_{1},\,u_{2}\in[-T,T]$,
\begin{equation}
\limsup_{N\to\infty}\frac{\mathbb{E}\left\{ \prod_{i=1,2}\left|\det\left(\mathbf{A}_{N-1}^{\left(i\right)}+\mathbf{B}_{N-1}^{\left(i\right)}\right)\right|\indic_{\mathcal{E}^{c}}\right\} }{\prod_{i=1,2}\mathbb{E}\left\{ \left|\det\left(\mathbf{A}_{N-1}^{\left(i\right)}\right)\right|\right\} }=0,\label{eq:ABC1-2}
\end{equation}
then (\ref{eq:ubDelta}) follows and the proof is completed for the
same $p$.

By H{\"{o}}lder's inequality, for $k>2$ the numerator above is bounded
by
\begin{equation}
  \label{eq-Holder}
\prod_{i=1,2}\mathbb{E}\left\{ \left|\det\left(\mathbf{A}_{N-1}^{\left(i\right)}+\mathbf{B}_{N-1}^{\left(i\right)}\right)\right|^{k}\right\} ^{1/k}\P\{\mathcal{E}^{c}\}^{1/k'},
\end{equation}
where $\frac{2}{k}+\frac{1}{k'}=1$, i.e. $\frac{1}{k'}=\frac{k-2}{k}$.
By Lemma \ref{lem:detMomentBd}, for some constant $c$ and even $k$,
\[
\prod_{i=1,2}\frac{\mathbb{E}\left\{ \left|\det\left(\mathbf{A}_{N-1}^{\left(i\right)}+\mathbf{B}_{N-1}^{\left(i\right)}\right)\right|^{k}\right\} ^{1/k}}{\mathbb{E}\left\{ \left|\det\left(\mathbf{A}_{N-1}^{\left(i\right)}\right)\right|\right\} }\leq cN^{k-1}.
\]

Hence, (\ref{eq:ABC1-2}) will follow if
\begin{equation}
\limsup_{N\to\infty}\P\{\mathcal{E}^{c}\}N^{\frac{k(k-1)}{k-2}}=0,\label{eq:Ec}
\end{equation}
uniformly in the stated parameters.
From Lemma \ref{lem:evgap} and a union bound, uniformly in
the stated parameters,
\[
N^{\frac{k(k-1)}{k-2}}\P(\mathcal{E}_1^c)\leq
2\left(20N^{1-\alpha} + 10\sum_{\ell=1}^{\lceil\log_2N^\alpha\rceil-1} N^{-\beta}\right)N^{\frac{k(k-1)}{k-2}}\to_{N\to\infty} 0,
\]
assuming that
\begin{equation}
	\label{eq:cond2}
	\min\{\alpha-1, \beta\}>\frac{k(k-1)}{k-2}.
\end{equation}

Similarly, from Lemma \ref{lem:B_ub}, uniformly in the stated parameters,
\[N^{\frac{k(k-1)}{k-2}}\P(\mathcal{E}_2^c)\to_{N\to\infty} 0,\]
whenever
\begin{equation}
	\label{eq:cond3}
	\frac{p-3}{4}-\gamma>0.
\end{equation}

For $k=4$, $\alpha=7+\epsilon$, $\beta=6+\epsilon$ and $\gamma=7+2\epsilon$, the
inequalities \eqref{eq:cond1}, \eqref{eq:cond2} and \eqref{eq:cond3} hold  for small $\epsilon$ and
any $p\geq32$. This completes the proof.\qed

\subsubsection{Proof of Lemma \ref{lem:B_ub}}

Define $\hat{\mathbf{B}}_{N-1}^{\left(i\right)}$ as the matrix obtained
from $\mathbf{B}_{N-1}^{\left(i\right)}$ by replacing by $0$ all
the elements in the $N-2\times N-2$ upper-left block. Recall that
the $N-2\times N-2$ upper-left block of $\mathbf{B}_{N-1}^{\left(i\right)}$
is $\mathbf{G}_{N-2}^{\left(i\right)}\left(r\right)-\mathbf{G}_{N-2}^{\left(i\right)}\left(0\right)$.
Of course,
\[
\max_{j\leq N-1}|\lambda_{j}(\mathbf{B}_{N-1}^{\left(i\right)})|\leq\max_{j\leq N-1}|\lambda_{j}(\hat{\mathbf{B}}_{N-1}^{\left(i\right)})|+\max_{j\leq N-1}|\lambda_{j}(\mathbf{G}_{N-2}^{\left(i\right)}\left(r\right)-\mathbf{G}_{N-2}^{\left(i\right)}\left(0\right))|.
\]
Hence, by a union bound, to prove the lemma it will be sufficient
to prove under the same conditions that
\begin{align}
\P\left\{ \max_{j\leq N-1}|\lambda_{j}(\hat{\mathbf{B}}_{N-1}^{\left(i\right)})|\geq t_{N}\right\}  & \leq e^{-KN},\label{eq:lmax1}\\
\P\left\{ \max_{j\leq N-1}|\lambda_{j}(\mathbf{G}_{N-2}^{\left(i\right)}\left(r\right)-\mathbf{G}_{N-2}^{\left(i\right)}\left(0\right))|\geq t_{N}\right\}  & \leq e^{-KN}.\label{eq:lmax2}
\end{align}

Assume $p\geq3$. Recall the definition of $\mathbf{G}_{N-2}^{\left(i\right)}\left(r\right)$
from Section \ref{sec:Conditional-distributions}, using the i.i.d.
matrices $\bar{\mathbf{G}}=\bar{\mathbf{G}}_{N-2}$, $\bar{\mathbf{G}}^{\left(1\right)}=\bar{\mathbf{G}}_{N-2}^{\left(1\right)}$,
and $\bar{\mathcal{\mathbf{G}}}^{\left(2\right)}=\bar{\mathcal{\mathbf{G}}}_{N-2}^{\left(2\right)}$.
Note that, for small $r$,
\begin{align*}
\mathbf{G}_{N-2}^{\left(i\right)}\left(r\right)-\mathbf{G}_{N-2}^{\left(i\right)}\left(0\right) & =\Big(\sqrt{1-\left|r\right|^{p-2}}-1\Big)\bar{\mathbf{G}}^{\left(i\right)}\pm\sqrt{\left|r\right|^{p-2}}\bar{\mathbf{G}}\\
 & =\Big(-\frac{1}{2}\left|r\right|^{p-2}+O(\left|r\right|^{2(p-2)})\Big)\bar{\mathbf{G}}^{\left(i\right)}\pm\sqrt{\left|r\right|^{p-2}}\bar{\mathbf{G}}.
\end{align*}
And in distribution,
\[
\mathbf{G}_{N-2}^{\left(i\right)}\left(r\right)-\mathbf{G}_{N-2}^{\left(i\right)}\left(0\right)\overset{d}{=}f(r)\bar{\mathbf{G}},
\]
where $f(r)=\sqrt{\left|r\right|^{p-2}}(1+o(1))$.

Hence, by \cite[Lemma 6.3]{BDG}, denoting by $\lambda_{i}(\mathbf{T})$
the eigenvalues of a matrix $\mathbf{T}$,
\[
\P\left\{ \max_{j\leq N-2}\Big|\lambda_{j}\Big(\mathbf{G}_{N-2}^{\left(i\right)}(r)-\mathbf{G}_{N-2}^{\left(i\right)}(0)\Big)\Big|\geq t\right\} \leq2\exp\left\{ -\frac{(N-2)t^{2}}{10\left|r\right|^{p-2}}\right\} ,
\]
for all $r\in(-r_{0},r_{0})$, for some $r_{0}>0$ independent of
$N$. From this bound, (\ref{eq:lmax2}) follows.

Working with the expressions in Section \ref{sec:Conditional-distributions},
one can verify that
\begin{align*}
\left|1-\sqrt{\frac{\Sigma_{Z,11}\left(r\right)-\left|\Sigma_{Z,12}\left(r\right)\right|}{p(p-1)}}\right| & \leq C|r|^{p-3},\\
\sqrt{\frac{\left|\Sigma_{Z,12}\left(r\right)\right|}{p(p-1)}} & \leq C|r|^{\frac{p-3}{2}},\\
\left|1-\sqrt{\frac{\Sigma_{Q,11}\left(r\right)-\left|\Sigma_{Q,12}\left(r\right)\right|}{2p(p-1)}}\right| & \leq C|r|^{p-4},\\
\sqrt{\frac{\left|\Sigma_{Q,12}\left(r\right)\right|}{2p(p-1)}} & \leq C|r|^{\frac{p-4}{2}},\\
\frac{m_{i}\left(r,u_{1},u_{2}\right)}{\sqrt{p(p-1)}} & \leq CT|r|^{p-2},
\end{align*}
for some $C=C(p)$ and any $u_{1},u_{2}\in[-T,T]$ and $r\in(-r_{0},r_{0})$
for small enough $r_{0}>0$, independent of $N$. From these bounds,
for any $u_{1},u_{2}\in[-T,T]$ and $r\in(-r_{0},r_{0})$, the Frobenius
norm $\|\hat{\mathbf{B}}_{N-1}^{\left(i\right)}\|$ is stochastically
dominated by
\[
\frac{4C|r|^{\frac{p-3}{2}}}{\sqrt{N-1}}\chi_{N-2}+\frac{4C|r|^{\frac{p-4}{2}}}{\sqrt{N-1}}|Y|+\frac{CT|r|^{p-2}}{\sqrt{N-1}},
\]
where $\chi_{N-2}$ is a chi variable of degree $N-2$ and $Y$ is
an independent standard Gaussian variable. By a union bound, (\ref{eq:lmax2})
easily follows.\qed

\subsubsection{Proof of Lemma \ref{lem:evgap}}
Using Markov's inequality we obtain that
\[\P\left\{ \mathbf{G}\text{ has at least $t$ eigenvalues in }I\right\}
\leq \frac{N\int_I \rho_N(x) dx}{t},\]
where $\rho_N(x)dx= \E L_N(dx)$. By e.g. \cite[(7.1.45)]{Forrester},
$\rho_N(x)$ converges uniformly  to the semicircle density,
which is bounded above. The claim follows (with the constant $10$ having no particular significance).
\qed

\subsection{\label{subsec:Combining-the-bounds}Proof of Theorem \ref{thm:main}}

The ratio in (\ref{eq:moment_matching}) is bounded from below by
$1$ for any $N$, so we only need to prove the matching upper bound
in the limit. By Lemma \ref{lem:large}, for some sequences of real
numbers $a_{N},\,b_{N}\to\min\{u,0\}$ and $\rho_{N}\to0$,
\[
\limsup_{N\to\infty}\frac{\mathbb{E}\big\{\left({\rm Crt}_{N}\left(\left(-\infty,u\right)\right)\right)^{2}\big\}}{\big(\mathbb{E}\left\{ {\rm Crt}_{N}\left(\left(-\infty,u\right)\right)\right\} \big)^{2}}=\limsup_{N\to\infty}\frac{\mathbb{E}\big\{\left[{\rm Crt}_{N}\left(\left(a_{N},b_{N}\right),(-\rho_{N},\rho_{N})\right)\right]_{2}\big\}}{\big(\mathbb{E}\left\{ {\rm Crt}_{N}\left(\left(a_{N},b_{N}\right)\right)\right\} \big)^{2}}.
\]

By Corollary \ref{cor:smalloverlaps}, the right-hand side above is
bounded by
\begin{align*}
 & \limsup_{N\to\infty}\sup_{u_1,u_2\in (a_{N},b_{N}) }\int_{(-\rho_{N},\rho_{N})\setminus(-r_{N},r_{N})}\sqrt{\frac{N}{2\pi}}e^{-\frac{1}{2}Nr^{2}+NB(r)}\Delta_{N}(r,u_{1},u_{2})dr\\
 & +\limsup_{N\to\infty}\sup_{u_1,u_2\in (a_{N},b_{N}) }\int_{(-r_{N},r_{N})}\sqrt{\frac{N}{2\pi}}e^{-\frac{1}{2}Nr^{2}+NB(r)}\Delta_{N}(r,u_{1},u_{2})dr,
\end{align*}
where we set   $r_{N}:=C\sqrt{\log N/N}$ with
some  constant $C>\sqrt 2$.

By Lemma \ref{lem:Intermediate} the first term above is equal to
$0$ and by Lemma \ref{lem:small_r} the second term is
bounded above by, and therefore equal to, $1$.
This completes the proof.\qed

\subsection{\label{sec:pfofcor1}Proof of Corollary \ref{cor:main}}  For $E>-E_\infty$ and for $E\in(-\Es,-E_\infty)$, the corollary follows directly from \eqref{eq-120521a} and \cite[Corollary 2]{2nd}, respectively. For $E=-E_\infty$, it follows by monotonicity of $u\mapsto {\rm Crt}_{N}((-\infty,u))$ and continuity of $\Theta_p(u)$. \qed

\appendix
\section{\label{sec:Conditional-distributions}Conditional distribution of
the Hessians}

In this appendix we describe, based on \cite{2nd},
the law of the random variables in Lemma
\ref{lem:KR2ndmoment}, which expresses the expectation of $\left[{\rm Crt}_{N}\left(B,I_{R}\right)\right]_{2}$.
Define
\begin{equation}
f_{N}\left(\boldsymbol{\sigma}\right)=f_{N,p}\left(\boldsymbol{\sigma}\right):=\frac{1}{\sqrt{N}}H_{N,p}\left(\sqrt{N}\boldsymbol{\sigma}\right).\label{eq:f-1}
\end{equation}
Given a (piecewise) smooth orthonormal frame field $E=\left(E_{i}\right)_{i=1}^{N-1}$
on $S^{N-1}:=\{\bx\in\R^{N}:\,\|\bx\|=1\}$, define
\[
\nabla f_{N}\left(\boldsymbol{\sigma}\right):=\left(E_{i}f_{N}\left(\boldsymbol{\sigma}\right)\right)_{i=1}^{N-1},\,\,\nabla^{2}f_{N}\left(\boldsymbol{\sigma}\right):=\left(E_{i}E_{j}f_{N}\left(\boldsymbol{\sigma}\right)\right)_{i,j=1}^{N-1}.
\]

Let $r\in\left(-1,1\right)$ and suppose that $\bs_{1}$ and $\bs_{2}$
are two points in $S^{N-1}$ such that $\bs_{1}\cdot\bs_{2}=r$. The
following was shown in  \cite[Lemma 13]{2nd}. First, conditional
on $\nabla f\left(\bs_{1}\right)=\nabla f\left(\bs_{2}\right)=0$,
the pair $(f_{N}(\boldsymbol{\sigma}_{1}),f_{N}(\boldsymbol{\sigma}_{2}))$
has the same distribution as $\left(U_{1}\left(r\right),U_{2}\left(r\right)\right)\sim N\left(0,\Sigma_{U}\left(r\right)\right)$,
where the matrix $\Sigma_{U}\left(r\right)$ is given by (\ref{eq:26})
below. Second, for an appropriate choice of $E=\left(E_{i}\right)_{i=1}^{N-1}$,
conditional on $f\left(\bs_{1}\right)=u_{1}$, $f\left(\boldsymbol{\sigma}_{2}\right)=u_{2}$
and $\nabla f\left(\bs_{1}\right)=\nabla f\left(\bs_{2}\right)=0$,
the pair
\[
\left(\frac{\nabla^{2}f\left(\bs_{i}\right)}{\sqrt{\left(N-1\right)p\left(p-1\right)}}\right)_{i=1,2}
\]
has the same law as
\begin{equation}
\left(\mathbf{X}_{N-1}^{\left(i\right)}(r)-\sqrt{\frac{1}{N-1}\frac{p}{p-1}}u_{i}I+\mathbf{E}_{N-1}^{\left(i\right)}(r,u_{1},u_{2})\right)_{i=1,2}\label{eq:XE}
\end{equation}
where the $N-1\times N-1$ matrices in (\ref{eq:XE}) are defined
as follows. The $N-1,N-1$ entry of $\mathbf{E}_{N-1}^{\left(i\right)}(r,u_{1},u_{2})$
is equal to $\frac{m_{i}\left(r,u_{1},u_{2}\right)}{\sqrt{(N-1)p(p-1)}}$,
(see (\ref{eq:m_i})) and all its other entries are zero. The matrices
$(\mathbf{X}_{N-1}^{\left(1\right)}(r),\mathbf{X}_{N-1}^{\left(2\right)}(r))$
are jointly Gaussian with block structure
\begin{align}
\mathbf{X}_{N-1}^{\left(i\right)}\left(r\right) & =\left(\begin{array}{cc}
\mathbf{G}_{N-2}^{\left(i\right)}\left(r\right) & Z^{\left(i\right)}\left(r\right)\\
\left(Z^{\left(i\right)}\left(r\right)\right)^{T} & Q^{\left(i\right)}\left(r\right)
\end{array}\right),\label{eq:ghat-1}
\end{align}
where:
\begin{enumerate}
\item The random elements $\left(\mathbf{G}_{N-2}^{\left(1\right)}\left(r\right),\mathbf{G}_{N-2}^{\left(2\right)}\left(r\right)\right)$,
$\left(Z^{\left(1\right)}\left(r\right),Z^{\left(2\right)}\left(r\right)\right)$,
and $\left(Q^{\left(1\right)}\left(r\right),Q^{\left(2\right)}\left(r\right)\right)$
are independent.
\item The matrices $\mathbf{G}^{\left(i\right)}\left(r\right)=\mathbf{G}_{N-2}^{\left(i\right)}\left(r\right)$
are $N-2\times N-2$ random matrices such that $\sqrt{\frac{N-1}{N-2}}\mathbf{G}^{\left(i\right)}\left(r\right)$
is a GOE matrix and, in distribution,
\[
\left(\begin{array}{c}
\vphantom{\left\{ \left\{ \left\{ \right\} ^{2}\right\} ^{2}\right\} ^{2}}\mathbf{G}^{\left(1\right)}\left(r\right)\\
\vphantom{\left\{ \left\{ \left\{ \right\} ^{2}\right\} ^{2}\right\} ^{2}}\mathbf{G}^{\left(2\right)}\left(r\right)
\end{array}\right)=\left(\begin{array}{c}
\vphantom{\left\{ \left\{ \left\{ \right\} ^{2}\right\} ^{2}\right\} ^{2}}\sqrt{1-\left|r\right|^{p-2}}\bar{\mathbf{G}}^{\left(1\right)}+\left({\rm sgn}\left(r\right)\right)^{p}\sqrt{\left|r\right|^{p-2}}\bar{\mathbf{G}}\\
\vphantom{\left\{ \left\{ \left\{ \right\} ^{2}\right\} ^{2}\right\} ^{2}}\sqrt{1-\left|r\right|^{p-2}}\bar{\mathbf{G}}^{\left(2\right)}+\sqrt{\left|r\right|^{p-2}}\bar{\mathbf{G}}
\end{array}\right),
\]
where $\bar{\mathbf{G}}=\bar{\mathbf{G}}_{N-2}$, $\bar{\mathbf{G}}^{\left(1\right)}=\bar{\mathbf{G}}_{N-2}^{\left(1\right)}$,
and $\bar{\mathbf{G}}^{\left(2\right)}=\bar{\mathcal{\mathbf{G}}}_{N-2}^{\left(2\right)}$
are independent and have the same law as $\mathbf{G}^{\left(i\right)}\left(r\right)$.
\item The column vectors $Z^{\left(i\right)}\left(r\right)=\left(Z_{j}^{\left(i\right)}\left(r\right)\right)_{j=1}^{N-2}$
are jointly Gaussian, and for $j\leq N-2$, $\left(Z_{j}^{\left(1\right)}\left(r\right),\,Z_{j}^{\left(2\right)}\left(r\right)\right)$
are i.i.d. with
\[
\left(Z_{j}^{\left(1\right)}\left(r\right),\,Z_{j}^{\left(2\right)}\left(r\right)\right)\sim N\left(0,\,\left(\left(N-1\right)p\left(p-1\right)\right)^{-1}\cdot\Sigma_{Z}\left(r\right)\right),
\]
where $\Sigma_{Z}\left(r\right)$ is given in (\ref{eq:sigmaZ}).
\item Lastly, $Q^{\left(i\right)}\left(r\right)$ are two Gaussian random
variables with
\[
\left(Q^{\left(1\right)}\left(r\right),\,Q^{\left(2\right)}\left(r\right)\right)\sim N\left(0,\,\left(\left(N-1\right)p\left(p-1\right)\right)^{-1}\cdot\Sigma_{Q}\left(r\right)\right),
\]
where $\Sigma_{Q}\left(r\right)$ is given in (\ref{eq:SigmaQ}).
\end{enumerate}
To define the covariances above, we first define, for any $r\in\left(-1,1\right)$,
\[
\begin{array}{ll}
a_{1}\left(r\right)=\frac{1}{p\left(1-r^{2p-2}\right)}, & a_{2}\left(r\right)=\frac{1}{p\left[1-\left(r^{p}-\left(p-1\right)r^{p-2}\left(1-r^{2}\right)\right)^{2}\right]},\\
a_{3}\left(r\right)=\frac{-r^{p-1}}{p\left(1-r^{2p-2}\right)}, & a_{4}\left(r\right)=\frac{-r^{p}+\left(p-1\right)r^{p-2}\left(1-r^{2}\right)}{p\left[1-\left(r^{p}-\left(p-1\right)r^{p-2}\left(1-r^{2}\right)\right)^{2}\right]},\\
b_{1}\left(r\right)=-p & b_{2}\left(r\right)=-pr^{p}\\
\,\,+a_{2}\left(r\right)p^{3}r^{2p-2}\left(1-r^{2}\right), & \,\,-a_{4}\left(r\right)p^{3}r^{2p-2}\left(1-r^{2}\right),\\
b_{3}\left(r\right)= & b_{4}\left(r\right)=p\left(p-1\right)r^{p-2}\left(1-r^{2}\right)\\
\,\,a_{2}\left(r\right)p^{2}\left(p-1\right)r^{2p-4}\left(1-r^{2}\right)\left[-\left(p-2\right)+pr^{2}\right], & \,\,-a_{4}\left(r\right)p^{2}\left(p-1\right)r^{2p-4}\left(1-r^{2}\right)\left[-\left(p-2\right)+pr^{2}\right].
\end{array}
\]

We define the covariance matrices $\Sigma_{U}\left(r\right)=\left(\Sigma_{U,ij}\left(r\right)\right)_{i,j=1}^{2,2}$,
$\Sigma_{Z}\left(r\right)=\left(\Sigma_{Z,ij}\left(r\right)\right)_{i,j=1}^{2,2}$
and $\Sigma_{Q}\left(r\right)=\left(\Sigma_{Q,ij}\left(r\right)\right)_{i,j=1}^{2,2}$
by
\begin{equation}
\Sigma_{U}\left(r\right)=-\frac{1}{p}\left(\begin{array}{cc}
b_{1}\left(r\right) & b_{2}\left(r\right)\\
b_{2}\left(r\right) & b_{1}\left(r\right)
\end{array}\right),\label{eq:26}
\end{equation}
by
\begin{equation}
\begin{aligned}\Sigma_{Z,11}\left(r\right) & =\Sigma_{Z,22}\left(r\right)=p\left(p-1\right)-a_{1}\left(r\right)p^{2}\left(p-1\right)^{2}r^{2p-4}\left(1-r^{2}\right),\\
\Sigma_{Z,12}\left(r\right) & =\Sigma_{Z,21}\left(r\right)=p\left(p-1\right)^{2}r^{p-1}-p\left(p-1\right)\left(p-2\right)r^{p-3}+a_{3}\left(r\right)p^{2}\left(p-1\right)^{2}r^{2p-4}\left(1-r^{2}\right),
\end{aligned}
\label{eq:sigmaZ}
\end{equation}
and
\begin{equation}
\begin{aligned}\Sigma_{Q,11}\left(r\right) & =\Sigma_{Q,22}\left(r\right)=2p\left(p-1\right)-a_{2}\left(r\right)\left(1-r^{2}\right)\left[p\left(p-1\right)r^{p-3}\left(pr^{2}-\left(p-2\right)\right)\right]^{2}\\
 & -\left(b_{3}\left(r\right),\,b_{4}\left(r\right)\right)\left(\Sigma_{U}\left(r\right)\right)^{-1}\left(\begin{array}{c}
b_{3}\left(r\right)\\
b_{4}\left(r\right)
\end{array}\right),\\
\Sigma_{Q,12}\left(r\right) & =\Sigma_{Q,21}\left(r\right)=p^{4}r^{p}-2p\left(p-1\right)\left(p^{2}-2p+2\right)r^{p-2}+p\left(p-1\right)\left(p-2\right)\left(p-3\right)r^{p-4}\\
 & +a_{4}\left(r\right)p^{2}r^{2p-6}\left(1-r^{2}\right)\left(p^{2}r^{2}-\left(p-1\right)\left(p-2\right)\right)^{2}\\
 & -\left(b_{1}\left(r\right)+b_{3}\left(r\right),\,b_{2}\left(r\right)+b_{4}\left(r\right)\right)\left(\Sigma_{U}\left(r\right)\right)^{-1}\left(\begin{array}{c}
b_{2}\left(r\right)+b_{4}\left(r\right)\\
b_{1}\left(r\right)+b_{3}\left(r\right)
\end{array}\right).
\end{aligned}
\label{eq:SigmaQ}
\end{equation}
Finally, we define
\begin{equation}
\begin{aligned}m_{1}\left(r,u_{1},u_{2}\right) & =\left(b_{3}\left(r\right),b_{4}\left(r\right)\right)\left(\Sigma_{U}\left(r\right)\right)^{-1}\left(u_{1},u_{2}\right)^{T},\\
m_{2}\left(r,u_{1},u_{2}\right) & =m_{1}\left(r,u_{2},u_{1}\right).
\end{aligned}
\label{eq:m_i}
\end{equation}
\bibliographystyle{alpha}
\bibliography{master}

\end{document}